\documentclass[a4paper]{article}
\usepackage[utf8]{inputenc}

\usepackage{graphicx} 
\usepackage{geometry}

\usepackage{amssymb,amsmath,amsthm}
\usepackage{mathtools}
\usepackage[dvipsnames]{xcolor}
\usepackage{yhmath}
\usepackage{enumerate}
\usepackage{xspace}
\usepackage{amsthm,amsfonts,amsxtra,amscd}
\usepackage[all]{xy}
\usepackage{verbatim}
\usepackage{imakeidx}
\makeindex
\usepackage{hyperref}
\usepackage{tikz-cd}
\usepackage{csquotes}
\usepackage{ragged2e}
\usepackage{blindtext}

\newcommand{\Vcan}{V^{\can}}

\DeclareMathOperator{\cont}{cont}

\newcommand{\Tan}{T}

\newcommand{\op}{\text{op}}

\DeclareMathOperator{\Aut}{Aut}

\newcommand{\Autpiu}[1]{\Aut^0_{#1}}
\DeclareMathOperator{\Aff}{Aff}

\theoremstyle{plain}
\newtheorem{lemma}{Lemma}[subsection]

\newtheorem{proposition}[lemma]{Proposition}
\newtheorem{corollary}[lemma]{Corollary}

\theoremstyle{definition}
\newtheorem{definition}[lemma]{Definition}
\newtheorem{remark}[lemma]{Remark}

\theoremstyle{remark}

\theoremstyle{plain}

\theoremstyle{definition}

\theoremstyle{remark}
\newtheorem{ntz}[lemma]{Notation}

\newcommand{\mC}{\mathbb C}

\newcommand{\mG}{\mathbb G}

\newcommand{\mZ}{\mathbb Z}

\newcommand{\calE}{\mathcal E} 
\newcommand{\calF}{\mathcal F}

\newcommand{\calO}{\mathcal O}

\newcommand{\calT}{\mathcal T}

\newcommand{\goF}{\mathfrak F}

\newcommand{\gob}{\mathfrak b}

\newcommand{\gog}{\mathfrak g}

\newcommand{\gom}{\mathfrak m}
\newcommand{\gon}{\mathfrak n}

\renewcommand{\geq}    {\geqslant}

\DeclareMathOperator{\Hom}  {Hom}

\DeclareMathOperator{\Der}  {Der}

\DeclareMathOperator{\Spec} {Spec}

\DeclareMathOperator{\Set}{Set}

\DeclareMathOperator{\id}{id}
\DeclareMathOperator{\univ}{univ}
\DeclareMathOperator{\triv}{triv}

\DeclareMathOperator{\can}{can}

\DeclareMathOperator{\Op}{Op}
\DeclareMathOperator{\Conn}{Conn}

\DeclareMathOperator{\ad}{ad\,}

\DeclareMathOperator{\Alg}{Alg}

\usepackage{xcolor}
\usepackage{soul}

\title{A note on the bundle underlying Opers}
\author{Luca Casarin}

\begin{document}

\maketitle

\begin{center}
    \textbf{Abstract}
    
    \smallskip
    \justifying
    \noindent In this note we write down a proof of the following well known fact, in order to make the literature more transparent. Let $\gog$ be a simple Lie algebra, then for any smooth curve $C$, the bundle underlying any $\gog$-Oper depends only on the curve and it is induced by the canonical $\Aut O$ bundle $\Aut_C$ on $C$.  
\end{center}

\tableofcontents

\section{Introduction}

The notion of Opers was introduced by Beilinson and Drinfeld in \cite{beilinson1991quantization} and by Drinfeld and Sokolov in \cite{drinfeld1984lie} in the case of the punctured disk. Since then, Opers have been intensively studied in the context of representation theory, integrable systems, and the geometric Langlands program. In particular, they are intimately related to the representation theory of the affine Kac-Moody algebras at the critical level, as shown in the works of Feigin and Frenkel \cite{feigin1992affine} and Frenkel and Gaitsgory \cite{frenkel2009local} and since the seminal work of Beilinson and Drinfeld \cite{beilinson1991quantization} Opers represented a crucial object in the proof of the Geometric Langlands conjectures carried out by Gaitsgory et. al. (see for instance \cite{2024proofii}).

Given a reductive group $G$ and a smooth curve $C$ over $\mC$, Opers are special kind of $G$-local systems on $C$, that is, $G$-torsors (also called $G$-bundles) with a connection. In particular, given a $G$-torsor $\goF$ on $C$, an Oper structure on it is the datum of a $B$-reduction (for a fixed Borel subgroup $B$) $\goF_B$ of $\goF$, together with a connection on $\goF$ satisfying certain properties (see Definition \ref{def:opersmoothcurve}). 

The goal of this note is to write down the proof of Proposition \ref{prop:uniqueopertorsor}, in which we show that the underlying $B$-torsor $\goF_B$ of an Oper is always isomorphic to a canonical $B$-torsor $\goF_0$ which depends only on the curve $C$.

This result is known as its statement can be found in \cite[3.1.7,3.1.8]{beilinson1991quantization} and \cite[3.1,3.2,3.3]{beilinson2005opers} for the $\mathfrak{sl}_2$ case and \cite[3.4]{beilinson2005opers} and \cite[3.1.9]{beilinson1991quantization} for the case of a general $\gog$. That said, we could not find a detailed proof of this important carachterization. The goal of this note is therefore to fill in these details that seem missing in the present literature.

For completeness we review various known constructions. We start by recalling the notion of torsor and some results on them and then move on to recall the definition of the groups $\Aut O$, ${\Autpiu{}} O$, Jet schemes and the canonical ${\Autpiu{}} O$-torsor on any smooth curve $C$ over $\mC$. Finally, we recall the definition of Opers for a simple Lie algebra $\gog$ and review some of their basic properties. Most of the latter is directly taken from \cite{BDchirali}.

\subsection{Acknowledgements} 

The author would like to thank Andrea Maffei for many useful discussions and suggestions.

\noindent\textbf{Funding:} The author was funded by national
PRIN Grants 2022S8SSW and 2022HMBTTL and by INFN - CSN4 (Commissione Scientifica Nazionale 4 - Fisica Teorica), MMNLP project.

\section{Preliminaries and notation}\label{app:classic}

We work over $\mC$, or any algebraically closed field of characteristic $0$. We fix an simple Lie algebra $\gog$, consider its associated group of adjoint type $G$, a Borel subgroup $B \subset G$ and fix a smooth connected curve $C$ over $\mC$. We will talk about (étale) \emph{coordinates} on smooth curves, by that we mean local functions $t \in \calO_C$ such that $\Omega^1_C = \calO_C dt$. We think of our geometric objects (schemes in particular) as functors $\Aff^{\op}_{\mC}\to \Set$, where $\Aff_{\mC}=\Alg_{\mC}^{\text{op}}$ is the category of affine schemes over $\mC$ and $\Alg_{\mC}$ is the category of commutative, unital algebras over $\mC$.

\subsection{Torsors}

We start by recalling the definition of a torsor and study $B$-torsors, for $B$ a Borel subgroup of a reductive group. In what follows we freely use the notion of fppf, smooth, étale and Zariski topology on $\Aff_{\mC}$.

\begin{definition}
	Let $H$ be any algebraic group. A (right) $H$-torsor on $C$ is a scheme $\goF$ on which $H$ acts on the right, equipped with a faithfully flat map of finite presentation $\goF \to C$ which is $H$-invariant and for which the morphism
	\[
		\goF \times H \to \goF \times_C \goF \qquad (s,h) \mapsto (s,sh)
	\]
	is an isomorphism. A torsor $\goF$ is said to be trivial if there exists an $H$-equivariant isomorphism $\goF \simeq C \times H$ over $C$. A torsor $\goF$ is said to be Zariski (resp. étale, smooth, fppf) locally trivial if there exists a Zariski (resp. étale, smooth, fppf) cover $C' \to C$ such that $\goF' = C'\times_C \goF$ is a trivial torsor.
\end{definition}

\begin{remark}\label{rmk:recollectiontorsor} We list some well known facts about torsors:
	\begin{enumerate}
		\item One may define the notion of a \emph{left} $H$-torsor in an analogous way, by requiring that the action of $H$ on $\goF$ is on the left. Any right torsor gives a left torsor by considering the left action of $H$ defined by $H \times \goF \to \goF, \quad (h,s) \mapsto sh^{-1}$. Similarly any left $H$ torsor induces a right torsor, so that these notions are essentially equivalent.
		\item The fact that an $H$ torsor $\goF$ is trivial is equivalent to the existence of a section $\sigma : C \to \goF$. Indeed the induced map $C\times H \to \goF \quad (x,h) \mapsto \sigma(x)h$ is an isomorphism after we take the pullback along $\goF \to C$ so that it is itself an isomorphism by faithfully flat descent. I follows that if $C' \to C$ is an fppf (resp. smooth, étale, Zariski) cover. Then the fact that $C' \times_C \goF$ is trivial is equivalent to the existence of a lift
		\[\begin{tikzcd}
	&& \goF \\
	\\
	{C'} && C
	\arrow[from=1-3, to=3-3]
	\arrow[dashed, from=3-1, to=1-3]
	\arrow[from=3-1, to=3-3]
\end{tikzcd}\]
		\item Let $\tau$ be any of the Zariski, étale, smooth, fppf topology on $C$. Then isomorphism classes of $\tau$-locally trivial torsors are classified by a pointed set $H^1_\tau(C,H)$ which may be computed à la \v{C}ech. 
		Given a short exact sequence of algebraic groups $1 \to K \to H \to Q \to 1$ for which $H \to Q$ is surjective morphism of sheaves for the $\tau$ topology, there is an attached exact sequence of pointed sets
		\[
			1 \to H^0(C,K) \to H^0(C,H) \to H^0(C,Q) \to H^1_\tau(C,K) \to H^1_\tau(C,H) \to H^1_\tau(C,Q),
		\]
		where $H^0(C,H)$ is the group of maps $C \to H$, with base point the trivial map $C \to e \to H$.
	\end{enumerate} 
\end{remark}

\begin{remark}\label{rmk:torsorsmoothtrivial}
	The requirement that $\goF \to C$ is faithfully flat and of finite presentation, together with the assumption that $\goF \times H \simeq \goF \times_C \goF$ is an isomorphism implies that any torsor is locally trivial for the fppf topology. If the group $H$ is smooth, which is our case since we are working with a field of characteristic $0$, then, by faithfully flat descent for smooth morphism, the map $\goF \to C$ is also smooth. This implies that $\goF$ is smooth-locally trivial as well.
\end{remark}

With the following Proposition we recall that any torsor for a smooth group is actually étale-locally trivial, in particular, in the discussion on torsors relative to smooth groups it is enough to treat the case of étale-locally trivial torsors.

\begin{proposition}\label{prop:torsoretaletrivial}
	Let $H$ be a smooth algebraic group and let $\goF \to C$ be an $H$ torsor. Then there exists a Zariski open cover $C = \bigcup_i U_i$ together with étale maps $U'_i \to U_i$ such that $U'_i \times_{U_i} \goF_{|U_i}$ is a trivial torsor. In particular $\goF$ is étale-locally trivial.
\end{proposition}

\begin{proof}
	By \cite[Example 2.51]{book2005fundamental} any smooth cover $C' \to C$ admits a refinement $C'' \to C' \to C$ for which $C'' \to C$ is an étale cover. It follows that if $C' \times_C \goF$ is trivial the same goes for $C'' \times_C \goF$. 
\end{proof}

\begin{lemma}\label{lem:extzariskitrivial}
	Let $1 \to K \to H \to Q \to 1$ be a short exact sequence of smooth groups. Assume that all $K$-torsors and $Q$-torsors are Zariski-locally trivial. Then any $H$-torsor is Zariski-locally trivial.
\end{lemma}

\begin{proof}
	Let $\goF$ be an $H$-torsor on $\goF$, which, by Proposition \ref{prop:torsoretaletrivial} we know to be étale-locally trivial. For any point $x \in C$ we want to find a Zariski open neighborhood $x \in U \subset C$ such that $\goF_{|U}$ is trivial. Consider the image $\goF_Q$ of $\goF$ along $H^1_{\text{ét}}(C,H) \to H^1_{\text{ét}}(C,Q)$, which by assumption is a Zariski-locally trivial torsor. We may find a Zariski-open neighborhood $x \in U'$ such that $\goF_Q$ is trivial on $U'$. By the exact sequence $H^1_{\text{ét}}(U',K) \to H^1_{\text{ét}}(U',H) \to H^1_{\text{ét}}(U',Q)$ it follows that $\goF_{|U'}$ comes from a $K$-torsor $\goF_K$. Since by assumption $\goF_K$ is Zariski locally trivial, we may find a Zariski open subset $x \in U$ on which $\goF_K$ is trivial. It follows that $\goF_{|U} = (\goF_K)_{|U}\times_K H$ is trivial on $U$ as well.
\end{proof}

We will apply the previous Lemma together with the following fact.

\begin{proposition}\label{prop:gmgazariskitrivial}
	The following hold:
	\begin{enumerate}
		\item Torsors for the group $\mG_m$ are Zariski-locally trivial;
		\item Torsors for the group $\mG_a$ are Zariski-locally trivial.
	\end{enumerate}
\end{proposition}

\begin{proof}
	This can be found in \cite[EXP. XI, Prop. 5.1]{SGA1}.
\end{proof}

\begin{corollary}\label{coro:btorsorzariskitrivial}
	Let $G$ be a split reductive group and $B \subset G$ a Borel subgroup. Then any $B$-torsor is Zariski-locally trivial.
\end{corollary}

\begin{proof}
	We apply Lemma \ref{lem:extzariskitrivial} together with Proposition \ref{prop:gmgazariskitrivial}. By the general theory of reductive groups, $B$ is an extension of a unipotent group by a torus. By Lemma \ref{lem:extzariskitrivial}, it suffices to show that torsors for tori and unipotent groups are Zariski-locally trivial. A torus is a product (and thus an extension) of copies of $\mG_m$, so the claim follows from point (1) of Proposition \ref{prop:gmgazariskitrivial}. Any smooth unipotent group is an extension of copies of $\mG_a$, so the claim in this case follows from point (2) of Proposition \ref{prop:gmgazariskitrivial}.
\end{proof}

\section{\texorpdfstring{$\Aut O$}{Aut O}, Jet schemes and all that}

We start by recalling the definition and first properties of the groups $\Aut O, {\Autpiu{}} O$ and then move on by recalling Jet schemes the canonical ${\Autpiu{}} O$-torsor on any smooth curve $C$ over $\mC$.

\subsection{The groups \texorpdfstring{$\Aut O$}{Aut O}, \texorpdfstring{${\Autpiu{}} O$}{Aut+ O}}

Let $O = \mC[[z]]$ and let $\gom \subset O$ be the ideal $z\mC[[z]]$. For any commutative $\mC$-algebra $R$ consider $O_R = R[[z]]$ and the ideal $\gom_R = zR[[z]]$. We write $\Aut O$ for the group functor $\Aut O(R) = \Aut_R^{\cont}(O_R)$, where the topology of $R[[z]]$ is that generated by the ideal $\gom_R$. \index{$\Aut O$} This can be shown to be an ind-scheme
. We move on and recall the subgroup $\Autpiu{} O \subset \Aut O$ and some of its properties.

\begin{definition}
    We consider the group schemes ${\Autpiu{}} O,\Autpiu{n} (O)$ for $n\geq2$ defined as the following functors on $\mC$-algebras:
    \begin{align*}
		{\Autpiu{}}O(R) &= \{ \rho \in \Aut_R^{\cont}(O_R) : \rho(\gom_R) \subset \gom_R \},\\
		{\Autpiu{n}} O(R) &= \{ \rho \in \Aut_R(O_R/\gom_R^n) : \rho(\gom_R) \subset \gom_R \}.
	\end{align*}
	There\index{${\Autpiu{}} O, \Autpiu{n} O$} are natural surjective (also on $R$ points) group homomorphisms $${\Autpiu{}} O \xrightarrow{\pi_n} \Autpiu{n} O, \quad \Autpiu{m} O \xrightarrow{\pi_{n,m}} \Autpiu{n} O$$ which present ${\Autpiu{}} O$ as the limit of the group schemes $\Autpiu{n} O$ along the maps $\pi_{n,m}$.
\end{definition}

By continuity any automorphism $\rho \in \Aut O(R)$ (resp. $\in \Autpiu{n} O$) is determined by the formal series $\rho(z) = \sum_{k\geq 0} \rho_k z^k$, where $\rho_k \in R$ (resp. $\rho(z) = \sum_{k\geq 0}^{n-1} \rho_kz^k$). The condition for this series to determine an element in ${\Autpiu{}} O (R)$ is $\rho_0 = 0, \rho_1 \in R^*$. Thus, we have bijections
    \[
        {\Autpiu{}} O (R) \simeq \left\{ \sum_{k\geq 1} \rho_kz^k : \rho_{1} \in R^* \right\} \quad \Autpiu{n} O (R) = \left\{ \sum^{n-1}_{k = 1} \rho_kz^k : \rho_{1} \in R^* \right\}.
    \]
    The group multiplication reads as
    \(
        \tau_1(z) \cdot \tau_2(z) = \tau_2(\tau_1(z)).
    \)

    We denote by $\Der O$ and $\Der^0 O$\index{${\Autpiu{}} O$!$\Der^0 O$} the Lie algebras of $\Aut O$ and ${\Autpiu{}} O$ respectively. We have $\Der O = \mC[[z]]\partial_z$,\index{$\Aut O$!$\Der O$} while $\Der^0 O$ can be identified with the subalgebra of derivations $f(z)\partial_z \in \Der O$ for which $f(z) \in \gom$.

\subsection{Jet schemes and the canonical \texorpdfstring{${\Autpiu{}} O$}{Aut+ O}-torsor}

\begin{definition} Given any functor of $\mC$ algebras $X : \Alg_{\mC} \to \Set$, its jet space $JX$ and its $n$th jet spaces are defined as functors by the formulas
\[
	JX(R) = X(R[[z]]) \qquad J_nX(R) = X(R[z]/z^n).
\]
if $X$ is a scheme of finite type then it can be shown that both $JX$ and $J_nX$\index{Jet schemes, $JX$, $J_nX$} are representable by schemes which are affine over $X$, in this case we call them the jet scheme and the $n$-th jet scheme respectively. There is a natural \emph{left} ${\Autpiu{}}O$ action on $JX$ for which the map $JX \to X$ is invariant.
\end{definition}

\begin{definition}
	Let $C$ be a smooth curve over $\mC$. We define $\Aut_C$\index{$\Aut_C$}, a subfunctor of $JC$, as follows:
	\[
		\Aut_C(R) = \left\{ x \in JC(R)=C(R[[z]]):\; \hat{x}^*\Omega^1_C \to \Omega^{1,\cont}_{R[[z]]/R} \text{ is an isomorphism} \right\}.
	\]
	Here $\hat{x}^*\Omega^{1}_C = \varprojlim_n x^*_{|\Spec R[z]/z^n}\Omega^1_{C}$, while $\Omega^{1,\cont}_{R[[z]]/R} = \varprojlim_n \Omega^1_{(R[z]/z^n)/R} = R[[z]]dz$. These are naturally $R[[z]]$ modules. $\Aut_C$ is a sheaf on the Zariski topology of $C$ and the left action of ${\Autpiu{}} O$ on $JC$ preserves this subfunctor.
\end{definition}

Assume that $C = \Spec A$ is an affine, smooth connected curve and consider an $R$-point $x : \Spec R \to C$. Let $J_x$ be the ideal of $A\otimes R$ defined as the kernel of the multiplication map $A \otimes R \xrightarrow{\text{mult}(x^*\otimes \,\_\,)} R$, so that $J_x$ is generated by the elements of the form $a \otimes 1 - 1\otimes x(a)$. We consider $\calT_x R$, the topological $R$-algebra defined by
\[
	\calT_x R = \varprojlim_n \frac{A\otimes R}{J_x^n}
\]

\begin{remark}
	If $C = \Spec A$ is affine with a coordinate $t \in A$ (so that $\Omega^1_C = \calO_C dt$) then there exists an isomorphism
	\[
		\rho^x_t : \calT_x R \to R[[z]] \qquad \text{such that } (\rho^x_t)^{-1}(z) = t - x(t)
	\]
\end{remark}

\begin{proof}
	It is enough to prove that the $R$-linear map $R[z] \to A\otimes R / J_x^n$ which sends $z$ to $t - x(t)$ induces an isomorphism $R[z]/z^n \to A\otimes R/J_x^n$ for any $n \geq 0$. This map is obtained by tensoring along $x^* : A \to R$ the universal map $A[z] \to A\otimes A/ I^n$ where $I = J_{\id} = \ker( \mathrm{mult} : A\otimes A \to A)$ for the universal $A$ point $\id : A \to A$. Then the result of the remark can be easily proved by induction on $n$, by using the fact that $I/I^2 = A\overline{t \otimes 1 - 1 \otimes t}$, which is equivalent to the assumption that $\Omega^1_C = \calO_C dt$.
\end{proof}

\begin{lemma}\label{lem:changecoordtR}
	Assume that $C = \Spec A$ is affine with coordinates $t,s \in A$, let $x : \Spec R \to C$ be a morphism. The composition
	\[
		\rho^x_{st} : R[[z]] \xrightarrow{(\rho^x_s)^{-1}} \calT_x R \xrightarrow{\rho^x_t} R[[z]],
	\]
	which is a continuous automorphism of $R[[z]]$, satisfies
	\[
		\rho^x_{st}(z) = \sum_{k\geq 1} x\left( \frac{1}{k!}\partial^k_t s \right)z^k.
 	\] 
\end{lemma}

\begin{proof}
	To show the lemma we need to compute the image of $s-x(s)$ along $\rho^x_t$. Notice that by construction $\rho^x_t$ intertwines with the actions of $\partial_t$ on $\calT_x R$ and $\partial_z$ on $R[[z]]$. In addition, for any $f \in R[[z]]$, if we write $\overline{f} = f(0) \in R$, it follows by the Taylor formula that $f = \sum_{k\geq 0} \frac{1}{k!}\overline{\partial^k_z f} z^k$. Recall also that by construction $\overline{\rho^x_t(g)} = x(g)$ for any $g \in A \subset R\otimes A$. From these remarks it follows that
	\begin{align*}
		\rho^x_t(s-x(s)) &= \sum_{k\geq 0} \frac{1}{k!}\overline{\partial^k_z \rho_t^x(s-x(s))} z^k = \sum_{k\geq 0} \frac{1}{k!}\overline{\rho_t^x(\partial_t^k(s-x(s)))} z^k \\ &= \sum_{k\geq 1} x\left( \frac{1}{k!}\partial^k_t s \right)z^k.
	\end{align*}
    The last equation comes from the fact that $\overline{\rho_t^x(s - x(s))} = \overline{\rho^t_x(s)} - x(s) = 0$ and that $\partial^k_t(s -x(s)) = \partial_t^ks$ for $k\geq 1$.
\end{proof}

\begin{proposition}
	The projection $\Aut_C \to C$ exhibits $\Aut_C$ as a left ${\Autpiu{}} O$-torsor for the Zariski topology over $C$. When $C$ is affine with a coordinate $t$, it's $R$ points may be described as follows:
    \[
        \Aut_C(R) = \left\{ (x,\varphi) \text{ such that } x\in C(R), \varphi: \calT_xR \xrightarrow{\simeq} R[[z]], \varphi(J_x) \subset zR[[z]]\right\}
    \]
    It follows that the choice of a local coordinate $t$ on $C$ induces a trivialization $\triv_t : C \times {\Autpiu{}} O \to \Aut_C$, which restricted to $C \times \{ 1 \}$ is given by $x \mapsto (x,\rho^x_t)$.
\end{proposition}

\begin{proof}
	The assertion is Zariski local on $C$ so we may assume that $C = \Spec A$ is affine with a chosen coordinate $t$ and reduce ourselves to construct an ${\Autpiu{}} O$ equivariant morphism $\triv_t : C \times {\Autpiu{}} O \to \Aut_C$ and to check that it is an isomorphism. Recall that to provide such a morphism is the same as to give a section $\triv_t : C \to \Aut_C$.

    It is enough to prove the description of $\Aut_C(R)$ which appears in the statement of the Proposition. All other statements follow, since the map
    \begin{align*}
        \triv_t : C \times {\Autpiu{}} O &\to \left\{ (x,\varphi):\; x\in C(R), \varphi: \calT_xR \xrightarrow{\simeq} R[[z]], \varphi(J_x) \subset zR[[z]]\right\} \\
        (x,\rho) &\mapsto (x,\rho\rho^x_t) 
    \end{align*}
    is clearly functorial, ${\Autpiu{}} O$-equivariant and bijective. We work fiberwise, fixing $x \in C(R) = \Hom(A,R)$ and considering the set $(\Aut_C(R))_x := \{ \tilde{\varphi} \in \Aut_C(R) \text{ with } \tilde{\varphi}_0 = x \}$. Where for $\tilde{\varphi} : A \to R[[z]]$ we set $\tilde{\varphi}_0$ to be the morphisms obtained from $\tilde{\varphi}$ by setting $z=0$. 

    We may describe $\Aut_C(R)$ as follows:
	\[
		\Aut_C(R) = \{ \tilde{\varphi} : A \to R[[z]] \text{ such that } (\partial_z\tilde{\varphi}(t)) \in (R[[z]])^* \}.
	\]
	Consider the unique $R$-linear extension $\tilde{\varphi}_R : R \otimes 
 A \to R[[z]]$ of $\tilde{\varphi}$. Notice that the condition $\tilde{\varphi}_0 = x$ translates into the condition that the ideal $J_x \subset R\otimes A$ is sent to $zR[[z]]$ under $\tilde{\varphi}_R$ which implies that $\tilde{\varphi}_R$ factors through the completion $A\otimes R \to \calT_x R \xrightarrow{\varphi} R[[z]]$ via a continuous $R$-morphism $\varphi$ which satisties $\varphi(J_x) \subset zR[[z]]$. Finally, since $\calT_xR$ is itself isomorphic to $R[[z]]$ via $\rho^x_t$, the condition $\partial_z\tilde{\varphi}(t) \in (R[[z]])^*$ is satisfied if and only if $\varphi$ is an isomorphism.
	
We obtain that $(\Aut_C(R))_x$ is in bijection with continuous $R$-isomorphisms $\varphi : \calT_xR \to R[[z]]$ such that $J_x \mapsto zR[[z]]$, as claimed.
\end{proof}

\begin{lemma}\label{lem:changecoordautx}
		Let $C = \Spec A$ be an affine smooth curve over $\mC$ and let $t,s$ be coordinates on $C$ (i.e. $t,s\in A$ such that $\Omega^1_C = Adt = Ads$).  Consider $\triv_{st} = \triv^{-1}_s\triv_t : C \times {\Autpiu{}} O \to C \times {\Autpiu{}} O$. This, being a morphism of ${\Autpiu{}} O$-torsors over $C$, is determined by an element $\triv_{st}^{\univ} \in {\Autpiu{}} O (A)$. We have
	\[
		\triv_{st}^{\univ}(z) = \sum_{k\geq 1} \frac{1}{k!}(\partial^k_ts)z^k \in A[[z]].
	\]
\end{lemma}

\begin{proof}
	By construction of $\psi_t$ on $R$ points we have
	\[
		\triv_{st} : (x,\rho) \mapsto (x,\rho\rho^x_{t}(\rho^x_{s})^{-1})
	\]
	and to compute where $\psi^{\univ}_{st}$ we just need to evaluate $\psi_{st}$ at the $A$-point $(\id_A,1)$. The statement of the Lemma then follows by Lemma \ref{lem:changecoordtR}.
\end{proof}

\section{Opers}

Let us set up some notation: here and in the remaining pages we fix a simple Lie algebra $\gog$, consider $G$, the algebraic group of adjoint type having $\gog$ as a Lie algebra and fix $B$, a Borel subgroup of $G$. In this Section we recall some facts about the space of Opers $\Op_{\gog}(C)$ for a simple Lie algebra $\gog$ and a smooth connected (not necessarily complete) curve $C$. Opers will be defined as couples $(\calF,\nabla)$, where $\calF$ is a $B$-torsor on a smooth curve $C$ and $\nabla$ is a connection on the induced $G$-torsor $\goF_G = \goF\times_B G$ satisfying some additional properties. Recall that, thanks to Corollary \ref{coro:btorsorzariskitrivial}, we may focus on treating Zariski-locally trivial torsors.

\subsection{Connections on smooth curves}

Let $C$ be a smooth curve, $G$ be an algebraic group over $\mC$.

\begin{ntz}\label{ntz:twistedspace} We will consider (right) $G$-torsors which are Zariski-locally trivial on $C$. Given such a $G$ torsor $\goF$ and a $G$-space $V$ we can consider the $\goF$-twisted $V$-bundle, denoted by $V_{\goF}$. This is constructed as the quotient space $\goF \times_G V$ along the diagonal $G$ action.  When $V$ is a finite dimensional vector space (i.e. a $G$ representation) $V_{\goF}$ is a vector bundle (i.e. a locally free $\calO_C$-module of finite rank).
\end{ntz}

Recall the Atiyah exact sequence of vector bundles on $C$ 
\[
	0 \to \gog_{\goF} \to \calE_{\goF} \to \Tan_C \to 0.
\]
Where $\Tan_C$ is the tangent bundle on $C$ and $\calE_{\goF}$ is the quotient of $T_\goF$, the tangent space on $\goF$, with respect to the natural $G$ action.

\begin{definition}\label{def:connection}
	A \textit{connection} on a $G$ torsor $\goF$ is a section $\Tan_C \to \calE_{\goF}$.\index{connections on a $G$ torsor} Connections naturally form a sheaf on the Zariski topology of $C$, to be denoted by $\Conn(\goF)$, which is naturally a $\gog_{\goF}\otimes \Omega^1_C$ torsor. 

	This definition works in families as well. Given any affine scheme $R$ an $R$-linear connection on $\goF_R$ is an $\calO_{C_R}$ linear section $\Tan_{C_R/R} \to \calE_{\goF_R}$ of exact sequence of locally free $\calO_{C_R}$-modules 
	\[
		0 \to \gog_{\goF_R} \to \calE_{\goF_R} \to \Tan_{X_R/R} \to 0
	\]
	obtained via pullback from the above one. 
\end{definition}

\begin{remark}\label{rmk:gaugeaction}
	If $\calF = C \times G$ is the trivial bundle, then there exists a canonical connection and $\Conn(C\times G) = \gog\otimes_{\mC}\Omega^1_C$. The sheaf of morphism $C \to G$ naturally acts, via the induced automorphism of the torsors $C \times G$, on $\Conn(C\times G)$. The induced action on $\gog\otimes_\mC \Omega^1_C$ is known as the \emph{Gauge action}\index{Gauge action}. For $g \in G(C), \omega \in \gog\otimes\Omega^1_C$ it is given by the following formula:
	\[
		g\cdot \omega = \text{Ad}_{g}(\omega) - dg\cdot g^{-1}.
	\]
\end{remark}

\subsection{Opers}

This material is taken from \cite[Ch. 3]{beilinson1991quantization} and \cite{beilinson2005opers}. Let $C$ be a smooth curve over $\mC$ and $G$ a simple algebraic group of adjoint type, fix a maximal torus and Borel subgroup $H \subset B \subset G$ let $N = [B,B]$. Consider $\gog$, the Lie algebra of $G$, with its root decomposition $\gog = \oplus_{\alpha \in \Phi} \gog^\alpha$ and fix a set of simple positive roots $\Delta = \{ \alpha_i \}$. Let $h_0$ be twice the sum of the fundamental coweights in $\mathfrak{h}$, so that $\ad h_0$ acts on each root space $\gog^\alpha$, for $\alpha = \sum n_i\alpha_i$ as $2|\alpha| = \sum 2n_i$. Fix also $f_0 = \sum_{i} f_{i}$ where $f_{i} \in \gog^{-\alpha_i}$ are non-zero elements. From this data it is possible to construct a unique principal $\mathfrak{sl}_2$-triple $\mathfrak{sl}_2 \simeq \{ e_0,h_0,f_0\} \subset \gog$. Since $G$ is adjoint, $\frac{1}{2}h_0$ may be exponentiated to a morphism $\check{\rho} : \mG_m \to H \to B$, the principal cocharacter, which induces the principal gradation on $\gog = \oplus_k \gog_k$, where $\gog_k = \oplus_{|\alpha| = k} \gog^\alpha$. We consider also $e:\mG_a \to B$, the exponential of $e_0$; this couple of morphisms induces a morphism $r : (B_2)_{\ad} \to B$, where $(B_2)_{\ad}$ \index{$(B_2)_{\ad}$}is the adjoint version of the group of upper triangular matrices $B_2 \subset SL_2$, so that 
\[
    (B_2)_{\ad} \simeq \left\{ \left( \begin{matrix} a & b \\ 0 & 1\end{matrix}\right) \right\}
\]

The action of $\check{\rho}$ induces a decreasing filtration of Lie algebras $\gog^{\geq k} \subset \gog, k \in \mZ$ such that $\gog^{\geq 1} = \gon, \gog^{\geq 0} = \gob$ and $\gog^{\geq k}/\gog^{\geq k+ 1} = \text{gr}^k \gog = \oplus_{|\alpha| = k} \gog^{\alpha}.$

Let $\goF_B$ be a $B$ torsor on $C$ and denote by $\goF_G$ the induced $G$ torsor. There is a commutative diagram of vector bundles on $C$
\[\begin{tikzcd}
	0 & {\gob_{\goF_B}} & {\calE_{\goF_B}} & {\Tan_C} & 0 \\
	0 & {\gog_{\goF_B}} & {\calE_{\goF_G}} & {\Tan_C} & 0
	\arrow[from=1-1, to=1-2]
	\arrow[from=1-2, to=1-3]
	\arrow[hook, from=1-2, to=2-2]
	\arrow[from=1-3, to=1-4]
	\arrow[hook, from=1-3, to=2-3]
	\arrow[from=1-4, to=1-5]
	\arrow["{=}"{description}, from=1-4, to=2-4]
	\arrow[from=2-1, to=2-2]
	\arrow[from=2-2, to=2-3]
	\arrow[from=2-3, to=2-4]
	\arrow[from=2-4, to=2-5]
\end{tikzcd}\]
Hence there is a natural inclusion $\Conn(\goF_B) \subset \Conn(\goF_G)$ which may be identified with the kernel of a natural map  $c : \Conn(\goF_G) \to (\gog/\gob)_{\goF_B}\otimes\Omega^1_X$. Given $\nabla \in \Conn(\goF_G)$, the section $c(\nabla)$ is constructed locally: on any open subset on which $\goF_B$ is trivial is defined by taking an arbitrary connection $\nabla_{\gob}$ on $\goF_B$ (which locally exists) and then computing $\nabla - \nabla_{\gob} \text{ mod } \gob\otimes\Omega^1_C$, which does not depend on the choice of $\nabla_{\gob}$. The morphism $c$ is equivariant with respect to the action of $\gog_{\goF}\otimes\Omega^1_C$ and therefore surjective as a morphism of sheaves on $C$.

\begin{definition}\label{def:opersmoothcurve}
	A $\gog$-oper\index{$\gog$-oper} is a couple $(\goF_B,\nabla)$ where $\goF_B$ is a $B$ torsor on $C$ and $\nabla$ is a connection of $\goF_G$ on $C$ such that 
	\begin{enumerate}
		\item $c(\nabla) \in (\text{gr}^{-1} \gog)_{\goF_B} \otimes \Omega^1_C \subset (\gog/\gob)_{\goF_B}\otimes\Omega^1_C$.
		\item For each negative simple root $\alpha$ the section $c(\nabla)^\alpha \in \Gamma(C,\gog^{\alpha}_{\goF} \otimes \Omega^1_C)$ never vanishes.
	\end{enumerate}
	There is an obvious notion of isomorphism of $\gog$-opers, we define $\Op_\gog(C)$ to be the groupoid of $\gog$-opers and isomorphisms between them. By \cite[1.3]{beilinson2005opers} a $\gog$-oper does not have any non-trivial automorphisms, it follows that $\Op_{\gog}(C)$ is equivalent to a set and that the assignment $U \mapsto \Op_{\gog}(U)$ is a sheaf of sets for the Zariski topology on $C$.
\end{definition}

\subsection{Canonical representative for Opers}\label{ssec:appendixcanonicalrepopers}

We recall here the description of $\Op_\gog(C)$ in the case where there exists a coordinate $t \in \calO_C$. 

Recall that we have a fixed $\mathfrak{sl}_2$ principal triple $\mathfrak{sl}_2 \simeq \{e_0.h_0,f_0\} \subset \gog$, which, by adjointess of $G$, induces a morphism $r : (B_2)_{\ad} \to G$ and consider $\mG_a \simeq N \subset (B_2)_{\ad}$. Let $\Vcan = \gog^N = \gog^{e_o}$, it is well known that $\check{\rho}$\index{$\Vcan$} acts on $\Vcan$ with eigenvalues $d_i$, where $d_i+1$ are the exponents of $\gog$. Fix also $x_i$ to be a basis of eigenvectors for $\Vcan$.

The following proposition goes back to \cite{drinfeld1984lie}.

\begin{proposition}\label{prop:descrlocopervcan}
	Let $(\goF,\nabla) \in \Op_\gog(C)$ be $\gog$-Oper. Assume that $C$ admits a function $t \in \calO_C$ such that $\Omega^1_C = \calO dt$ (i.e. a coordinate) and that $\goF$ is trivial. Then there is a unique element $\omega_\nabla \in \Vcan\otimes\Omega^1_{C} = \Vcan \otimes \calO_Cdt$ such that
	\[
		(\goF,\nabla) \simeq (C\times B, d + f_0dt + \omega_\nabla)
	\]
	as Opers. Since Opers do not have automorphisms, the above isomorphism is unique. This is true in families so that
	\[
		\Op_\gog(C)(R) \simeq \Vcan \otimes \Omega^1_{C_R/R}(C_R) =  \Vcan\otimes \calO(C)\otimes R dt.
	\]
\end{proposition}

Given $\omega_\nabla \in \Vcan\otimes\Omega^1_{C}$ and a coordinate $t \in \calO_C$ we will write $\omega_{\nabla}^t \in \Vcan \otimes \calO_C\otimes R$ for the element such that $\omega_{\nabla}^tdt = \omega_{\nabla}$ and $\omega_{\nabla}^{t,i} \in \calO_C\otimes R$ to denote the $i$-th component of $\omega_{\nabla}$ with respect to the basis $x_i$.

\begin{remark}\label{rmk:changecoordinateoper}
	Under the assumption that $C$ admits coordinates fix an oper $(\goF,\nabla) \in \Op_\gog(C)(R)$ for which $\goF$ is trivial. Given two coordinates $t,s \in \calO_C\otimes R$ consider the two unique elements $\omega_\nabla^t,\omega_\nabla^s$ such that
	\[
		(C\times B,d + (f + \omega_\nabla^t)dt) \simeq (\goF,\nabla) \simeq (C\times B, d + (f + \omega_\nabla^s)ds).
	\]
	Then
	\begin{align*}
		\omega^{s,1}_\nabla &= (\partial_st)^{2}\omega^{t,1}_\nabla -\frac{1}{2}\{t,s\}, \\
		\omega^{s,j}_\nabla &= (\partial_st)^{d_j+1}\omega^{t,1}_\nabla \quad j > 1,
	\end{align*}
	where $\{t,s\} = \left( \frac{\partial_s^3t}{\partial_st}- \frac{3}{2}\left(\frac{\partial^2_st}{\partial_st}\right)^2\right)$ is the \textit{Schwarzian Derivative}. The isomorphism above is provided by the element
	\[
		e\left(\frac{\partial^2_st}{2\partial_st}\right)\check{\rho}(\partial_s t) \in B(C).
	\]
\end{remark}

\begin{proof}
	Recall that the group $G(\calO_C(C)\otimes R)$ acts on the set of connections by Gauge transformations:
	\[
		g\cdot (d + A) = d + \text{Ad}_g(A) - dg\cdot g^{-1}.
	\]
	We have
	\[
		d + (f +\omega^t_\nabla)dt = d + (f\otimes\partial_st + \omega^t_\nabla\partial_st)ds
	\]
	and to bring it back into the canonical form we first need to act with $\check{\rho}(\partial_st)$, using the Gauge action we have
	\begin{align*}
		d + (f\otimes\partial_st + \omega^t_\nabla\partial_st)dt &\simeq \check{\rho}(\partial_st) \cdot \left( d + (f\otimes\partial_st + \omega^t_\nabla\partial_st)ds \right) = \\
		&= d + (f + \text{Ad}_{\check{\rho}(\partial_st)}(\omega^t_\nabla)\partial_st)ds -\frac{1}{2}h_o\otimes\frac{\partial^2_st}{\partial_st}ds
	\end{align*}
	We then act with $\mG_a \xrightarrow{e}  N \subset PSL_2 \to G$ to obtain
	\begin{align*}
		d + (f\otimes\partial_st + \omega^t_\nabla\partial_st)dt &\simeq e\left(\frac{\partial_s^2t}{2\partial_st}\right)\cdot \left(d + (f + \text{Ad}_{\check{\rho}(\partial_st)}(\omega^t_\nabla))ds -\frac{1}{2}h_o\otimes\frac{\partial^2_st}{\partial_st}ds \right) = \\
		&= d + (f + \text{Ad}_{\check{\rho}(\partial_st)}(\omega^t_\nabla)\partial_st)ds -\frac{1}{2} e_0 \otimes \left( \frac{\partial_s^3t}{\partial_st}- \frac{3}{2}\left(\frac{\partial^2_st}{\partial_st}\right)^2\right)ds.
	\end{align*}
	Since $e_0 = v_1$ the results follows.
\end{proof}

Using the description of canonical representatives we finally are able to give a proof of the following Proposition.

\begin{proposition}\label{prop:uniqueopertorsor}
	Let $r_O$ be the composition $r_O : {\Autpiu{}} O \to {\Autpiu{}}_3 (O) \simeq (B_2)_{\ad} \xrightarrow{r} B$ and consider \index{$\goF_0$}$\goF_0 = \Aut_C \times_{r_O} B$ which is a $B$-torsor on $C$. Let $(\goF,\nabla)$ be a $\gog$-oper on $C$. Then there exists an isomorphism $\goF \simeq \goF_0$.
\end{proposition}

\begin{proof}
    We compare Zariski cocycles defining $\goF$ and $\goF_0$. Notice that with our notation, $\goF$ is a \emph{right} $B$ torsor, while $\goF_0$ is \emph{left} torsor. We remedy this by considering $\goF$ as a left torsor via the inverse action, so that we have to take the inverse of its defining cocycle.  Fix an affine Zariski covering $(U_i)$ of $C$, for which each $U_i$ admits a coordinate $t_i$ and $\goF$ is trivial on $U_i$. It follows from Proposition \ref{prop:descrlocopervcan} that there are unique (Opers do not have automorphisms) trivializations $\gamma_i : U_i \times  B \to \goF$ such that $\nabla$ is read via $\gamma_i$ in its canonical form (that of Proposition \ref{prop:descrlocopervcan}). It follows from the proof of Remark \ref{rmk:changecoordinateoper} that a cocycle defining $\goF$ , $c_{ji} = \gamma_j^{-1}\gamma_i$, relative to the open cover $U_i$ and the isomorphisms $\gamma_i$, is given by the following:
    \[
        c_{ji} = e\left( \frac{\partial_{t_i}^2t_j}{2\partial_{t_i}t_j} \right)\check{\rho}(\partial_{t_i}t_j).
    \]
    This is the inverse of the element of $B$ providing the isomorphism of Remark \ref{rmk:changecoordinateoper} because of the passage between right and left torsors. It follows that this cocycle comes from the map $(B_2)_{\text{ad}} \to B$, recall that the group $(B_2)_{\text{ad}}$ is isomorphic to the group of matrices
    \[
        (B_2)_{\text{ad}} \simeq \left\{ \left( \begin{matrix} a & b \\ 0 & 1\end{matrix} \right) \right\}
    \]
    Via the identification ${\Autpiu{3}} \simeq (B_2)_{\ad}$
    \[
    (z \mapsto az + bz^2) \leftrightarrow \left(\begin{matrix}a & b/a \\ 0 & 1 \end{matrix} \right)
    \]
    The cocycle $c_{ji}$ corresponds to the $\Autpiu{3} O$ cocycle 
    \[
        c^{\Aut}_{ji}(z) = (\partial_{t_i}t_j)z + \left(\frac{1}{2}\partial_{t_i}^2t_j\right) z^2
    \]
    Which is exactly the cocycle defining the $\Autpiu{3} O$-torsor $\Aut_C \times_{{\Autpiu{}} O} {\Autpiu{3}} O$ as Lemma \ref{lem:changecoordautx} shows.
\end{proof}

\bibliography{biblio.bib}
\bibliographystyle{alpha}

\end{document}